\documentclass{article}

\usepackage{amssymb, amsmath, amsthm}

\usepackage{epsfig}
\usepackage{graphicx}
\usepackage{color}
\usepackage[shortalphabetic]{amsrefs}
\usepackage{mathptmx}
\usepackage{enumerate}
\usepackage{hyperref}
\usepackage{fullpage}
\usepackage{pinlabel}

\theoremstyle{plain}
\newtheorem{theorem}{Theorem}[section]

\newtheorem*{theorem*}{Theorem}

\newtheorem{proposition}[theorem]{Proposition}
\newtheorem{corollary}[theorem]{Corollary}
\newtheorem{lemma}[theorem]{Lemma}

\theoremstyle{definition}
\newtheorem{definition}[theorem]{Definition}

\newtheorem*{remark}{Remark}

\newtheorem*{conjecture}{Conjecture}

\theoremstyle{definition}

\newcommand{\R}{{\mathbb R}}

\newcommand{\N}{\mathbb N}

\newcommand{\A}{\mc{U}}
\newcommand{\B}{\mc{L}}

\newcommand{\nil}{\varnothing}

\newcommand{\defn}[1]{\emph{#1}}
\newcommand{\boundary}{\partial}
\newcommand{\mc}[1]{\mathcal{#1}}

\newcommand{\diam}{\operatorname{diam}} 

\renewcommand{\b}{\frak{b}}

\renewcommand{\u}{\frak{u}}
\renewcommand{\d}{\frak{d}}

\newcommand{\co}{\mskip0.5mu\colon\thinspace}

\begin{document}

   \title{Neighbors of knots in the Gordian graph}
   \author{Ryan Blair, Marion Campisi, Jesse Johnson, Scott A. Taylor, Maggy Tomova}

   \maketitle
   
   \begin{abstract}
The Gordian graph is the graph with vertex set the set of knot types and edge set consisting of pairs of knots which have a diagram wherein they differ at a single crossing. Bridge number is a classical knot invariant which is a measure of the complexity of a knot. It can be refined by another, recently discovered, knot invariant known as ``bridge distance''. We show, using arguments that are almost entirely elementary, that each vertex of the Gordian graph is adjacent to a vertex having arbitrarily high bridge number and bridge distance. 
\end{abstract}



\section{Introduction}

Understanding the effect of crossing changes on knots is a central endeavor of knot theory. For example, each knot $K$ can be transformed into the unknot by changing crossings, but the minimal number $\u(K)$ of such crossing changes needed (the \defn{unknotting number} of the knot), though venerable, is poorly understood. Another classical knot invariant, the bridge number $\b(K)$ (see Definition \ref{Def: Bridge}), is much better understood. Might there be some relation between them? Sadly, some familiarity with examples shows that they are probably unrelated. For example, a $(p,q)$-torus knot (that is, a knot lying on an unknotted torus $T$ wrapping $p$-times meridionally and $q$ times longitudinally, with $p$ and $q$ relatively prime) has $\u(K) = (p-1)(q-1)/2$ \cite{KM, R}, but $\b(K) = \min(p,q)$ \cite{Schubert, Schultens}. In particular, there are knots with fixed bridge number and arbitrarily large unknotting number. On the other hand, there are also knots (the \emph{iterated Whitehead doubles}) with unknotting number equal to 1 and bridge number arbitrarily large. These examples are all very special, so we are still left with the questions: How might we construct knots of given unknotting number and a given bridge number? How might we construct  infinitely many knots of a given unknotting number and a given bridge number?

The \defn{Gordian Graph}\footnote{The name pays homage to the well-known story of Alexander the Great slicing the Gordian knot.} is an object which organizes knots by the effect of crossing changes. It is the undirected graph with vertex set the set of knot types and with edge set consisting of pairs of distinct knot types which are related by a single crossing change. We say that a pair of knot types defining an edge are \defn{neighbors in the Gordian graph}. Thus, the neighbors of the unknot are precisely the knots $K$ with $\u(K) = 1$ and, in general, $\u(K)$ is the minimum number of edges in an edge path in the Gordian Graph from $K$ to the unknot.

Our main tool for studying the distribution of bridge numbers in the Gordian Graph is a relatively new, but very powerful, natural number knot invariant known as \emph{bridge distance}. The \defn{bridge distance} of a knot $K$ is denoted $\d(K)$ (see Definition \ref{Def: bridge distance}.) It is an invariant which allows us to distinguish between knots having the same bridge number and is larger the more complicated a knot is. Since the unknotting number is also a measure of the complexity of a knot, we might hope that there is some connection between the unknotting number and bridge distance. We prove, however, that this is not the case.

\begin{theorem}\label{Main Theorem}
For every knot $K$ and every $(b,n) \in \N^2$ with $b \geq \max(3, \b(K))$, there exists a knot $K'$ differing from $K$ by a single crossing change such that $\b(K') = b$ and $\d(K') \geq n$. \end{theorem}

In particular, every knot has neighbors in the Gordian graph of arbitrarily high bridge number and arbitrarily high bridge distance.

\subsection{Notation, Terminology, and Conventions}
If $X$ is a topological space, then $|X|$ will denote the number of connected components of $X$. If $X$ is a manifold-with-boundary, then $\boundary X$ denotes its (possibly empty) boundary.

Given topological spaces $X$ and $Y$ and a function $f \co X \to Y$ which is a homeomorphism onto its image, an isotopy of $f$ is a continuous function $F \co X \times [0,1] \to Y$ such that, for every $x \in X$, $F(x,0) = f(x)$ and for every $t \in [0,1]$, the function $f(\cdot, t) \co X \to Y$ is a homeomorphism onto its image. If  $X$ is a subspace of $Y$, then an \defn{isotopy} of $X$ in $Y$ is an isotopy of the inclusion map.  The isotopy is an \defn{ambient isotopy} if it can be extended to an isotopy of $Y$ in itself. All of the isotopies of subspaces appearing in this paper will be ambient isotopies (and the spaces $X$ and $Y$ will be clear from the context) so, for simplicity, we will simply refer to ambient isotopies as isotopies. If $X$ and $Y$ are both manifolds and if $X \subset Y$, we say that $X$ is \emph{properly embedded} in $Y$ if it is a submanifold of $Y$ and if $X \cap \boundary Y = \boundary X$. An isotopy of $X$ is a \defn{proper isotopy} if $X$ is properly embedded in $Y$ at each point of time. The isotopy is \defn{relative to $\boundary X$} if, for every $t$, we have $f(\boundary X, t)  = \boundary X$.

A \defn{knot} is a smooth simple closed curve embedded  in $S^3 = \R^3 \cup \{\infty\}$ and a link is the union of 1 or more pairwise disjoint knots. Two links are \defn{equivalent} if there is an (ambient) isotopy taking one to the other. This is an equivalence relation on links and a link's equivalence class is known as its \defn{type}.  A \defn{knot invariant} is an algebraic object (e.g. a number, vector space, polynomial, or group) associated to knot types. The knot invariants in this paper will all be functions from a subset of the set of knot types to the integers.

In this paper, a \defn{surface} will always be the result of removing a finite set (possibly the empty set) of points from the interior of a compact, orientable 2-dimensional manifold $F$. We will call the points removed from $F$ \defn{punctures}. If the surface is a subset of $S^3$, we will assume that it is smooth.

If $X$ and $Y$ are two surfaces in $S^3$, two smooth simple closed curves in a surface, or a surface and a knot in $S^3$, then we will always assume that $X$ and $Y$ intersect transversally. This implies that they have no points of tangency and the intersection $X \cap Y$ is a compact submanifold of the appropriate dimension. (If $X$ and $Y$ are surfaces in $S^3$, then $X \cap Y$ is the union of finitely many simple closed curves; if $X$ and $Y$ are simple closed curves in a surface then $X \cap Y$ is the union of finitely many points; if $X$ is a surface and $Y$ is a knot in $S^3$, then $X \cap Y$ is also the union of finitely many points.) If $Y$ is a subsurface of $X$, we will let $X \setminus Y$ denote the complement of the interior of $Y$, so that $X \setminus Y$ is either empty or is a surface.

An \defn{arc} $\alpha$ in a surface $F$ is a compact, properly embedded, connected, smooth 1-manifold with non-empty boundary (necessarily homeomorphic to the interval $[0,1]$). A \defn{curve} in $F$ is a smooth simple closed curve in $F$ disjoint from $\boundary F$. If $\alpha$ and $\beta$ are arcs or curves on a surface $F$, we let $i(\alpha,\beta)$ denote the minimum of $|\alpha' \cap \beta'|$ where $\alpha'$ and $\beta'$ range over all arcs or simple closed curves which are (ambiently) isotopic in $F$ to $\alpha$ and $\beta$ respectively and which are transverse to each other. If $\alpha$ and $\beta$ are the unions of arcs and curves in a surface $F$, then they intersect \defn{minimally}, if they are transverse and if for each component $\alpha' \subset \alpha$ and $\beta' \subset \beta$, we have $|\alpha' \cap \beta'| = i(\alpha',\beta')$.

Finally, throughout the paper, we will use without comment basic facts about simple closed curves on surfaces. The paper \cite{Epstein} and the book \cite{FM} are good sources for these facts. For example, we will use the (smooth) Sch\"onflies Theorem which states that every simple closed curve on a sphere bounds a disc on both sides. One other fact deserves particular attention (see \cite[Lemma 3.3]{FM}): if $\gamma_1, \hdots, \gamma_{n}$ are curves or arcs in a surface $F$ which pairwise intersect minimally and if $\gamma_{n+1}$ is another curve or arc, then there is an (ambient) isotopy of $\gamma_{n+1}$ in $F$ so that $\gamma_{n+1}$ is transverse to each of $\gamma_1, \hdots, \gamma_n$ and intersects them each minimally.

\section{Crossing Changes}

Changing a crossing is one of the most basic methods of changing the knot type of a knot $K$. One way of specifying a crossing change of $K$ is to choose some diagram for $K$ (not necessarily a diagram with the fewest possible crossings), choose a crossing in the diagram, and reverse the over and under strands, as in Figure \ref{SimpleCC}. Equivalently, we may choose a diagram of $K$ in which there are two parallel strands as on the left of Figure \ref{DiagramCC}, we then pull one strand over the other, as in the middle of Figure \ref{DiagramCC}, introducing two crossings, and, finally we reverse the over, under strands of one of the crossings, as on the right of Figure \ref{DiagramCC}.  For more details, on crossing changes in diagrams, see \cite[Section 3.1]{Adams}.

\begin{figure}[h!]
\centering
\includegraphics[scale = .5]{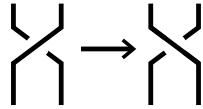}
\caption{A typical local picture of a crossing change}
\label{SimpleCC}
\end{figure}

\begin{figure}[h!]
\centering
\includegraphics[scale = .5]{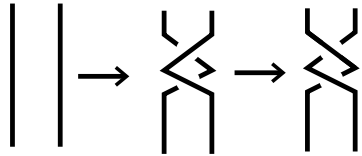}
\caption{A different diagrammatic depiction of a crossing change}
\label{DiagramCC}
\end{figure}

As intuitive as these diagrammatic descriptions are, a definition of crossing change which allows the knot to remain in 3-dimensional space will be more useful to us.

\begin{definition}
Suppose that $K \subset S^3$ is a knot. A \defn{crossing disc} for $K$ is a smoothly embedded disc $D \subset S^3$, such that $D$ intersects $K$ transversally in precisely two interior points of the disc. The boundary of the disc is called a \defn{crossing link} for $K$. A \defn{crossing change on $K$ using $D$} is a smooth homeomorphism $h \co S^3 \to S^3$ obtained (informally) by cutting $S^3$ open along $D$, twisting one of the resulting copies of $D$ by $2\pi$ and then re-gluing, as in Figure \ref{Twisting disc}. This converts the knot $K$ into a knot $K'$ which we say is \defn{obtained by a (topological) crossing change} on $K$ using $D$. Clearly, $K'$ may have a different type than $K$.
\end{definition}

A crossing change is a specific instance of ``Dehn surgery'' - a popular and well-studied operation in low-dimensional topology. In particular, a crossing change is a $\pm 1$ Dehn surgery on the crossing link for $K$. See, for example, \cite[Chapter 9]{Rolfsen}.

\begin{figure}[h!]
\centering
\includegraphics[scale = .5]{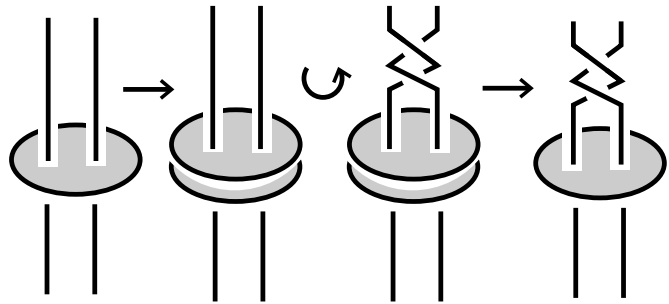}
\caption{A crossing change effected by twisting about a disc.}
\label{Twisting disc}
\end{figure}

\newpage

Here are some elementary facts concerning crossing changes:
\begin{itemize}
\item There are two possible choices for twisting the disc $D$ before regluing. Typically one of them is called a $+1$ crossing change along $D$ and the other is a $-1$ crossing change. Given an orientation of the crossing disc, there is an established convention for determining which is which, but we will not need it for this paper.
\item Each topological crossing change can be realized as a diagrammatic crossing change. (The proof is similar to the discussion on \cite[page 58]{Adams}.)
\item The knot type of $K'$ depends on only on the link type of $K \cup \boundary D$ and the direction of the twist.
\end{itemize}

\section{Bridge Spheres}

\subsection{Basic definitions and properties}

Informally, a bridge sphere is a sphere which cuts a knot into unknotted pieces. To formalize this, we need need to define trivial tangles:

\begin{definition}
A \defn{trivial tangle} $(B, \kappa)$ is a 3-ball $B$ containing properly embedded arcs $\kappa$ such that, fixing the endpoints of $\kappa$, we may isotope $\kappa$ into $\boundary B$. We also say that the arcs $\kappa$ are \defn{unknotted} in $B$. Figure \ref{TrivialTangle} shows a prototypical trivial tangle with $|\kappa| = 3$.
\end{definition}

\begin{figure}[h!]
\centering
\includegraphics[scale = .5]{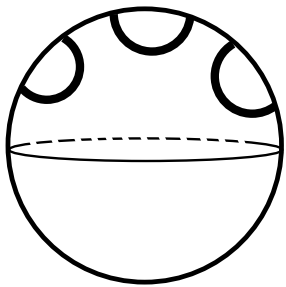}
\caption{A trivial tangle $(B,\kappa)$ with $|\kappa| = 3$. The arcs $\kappa$ are drawn with a thicker line.}
\label{TrivialTangle}
\end{figure}

Trivial tangles are, up to homeomorphism of pairs, determined completely by $|\kappa|$. We record this fact in the next lemma, whose proof we omit. (It can be proved by induction on $|\kappa|$ along the lines of Lemma \ref{Disjt Homeo}, below.)

\begin{lemma}\label{Pair Homeo}
If $(B, \kappa)$ and $(B',\kappa')$ are both trivial tangles with $|\kappa| = |\kappa'|$, then there is a homeomorphism of pairs $(B,\kappa) \to (B', \kappa')$.
\end{lemma}

If $(B,\kappa)$ and $(B', \kappa')$ are both trivial tangles with $|\kappa| = |\kappa'|$, we may construct a knot or link $K$ in $S^3$ by choosing a homeomorphism of pairs
\[h\co (\boundary B, \boundary \kappa) \to (\boundary B', \boundary \kappa')\] and then using the homeomorphism to glue the boundaries of the 3-balls and the arcs together:
\[
(S^3, K) = (B \cup_h B', \kappa\cup_h \kappa').
\]

It turns out that every knot type in $S^3$ can be built in such a way.  To see this, we recall that every knot has a diagram in the plane which is the plat closure of a braid, as in Figure \ref{Plat}. That is, with respect to projection onto a vertical line, the diagram has $b$ maxima and $b$ minima with all the maxima above all the minima. In between the maxima and minima, the strands of the knot are monotonically increasing or decreasing. We think of the knot $K$ itself as lying in the plane of projection except at the crossings where it dips slightly in front of and slightly behind the plane. A line in the plane of projection which separates the maxima from the minima can be extended orthogonally from the plane to a plane $P$ separating the maxima from the minima, shown as a thick dashed line in Figure \ref{Plat}. Adding a point at infinity to $\R^3$ produces the 3-sphere $S^3$ and converts our plane to a 2-sphere, which we continue to denote by $P$.  The bridge sphere $P$ separates $S^3$ into two 3-balls. The strands of $K$ lying inside those 3-balls can be isotoped, keeping their endpoints fixed into $P$, since each arc of $K\setminus P$ has only a single maximum or minimum. We formalize an equivalent construction in Definition \ref{Def: Bridge}. For a different perspective on bridge number see \cite[Section 3.2]{Adams}.

\begin{figure}[h!]
\centering
\includegraphics[scale = .5]{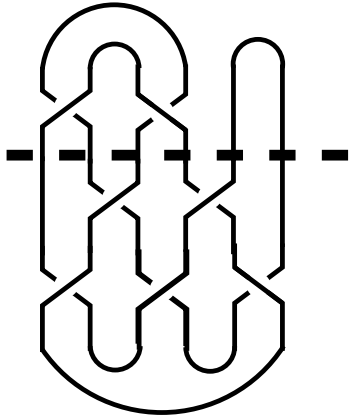}
\caption{A knot as a plat closure.}
\label{Plat}
\end{figure}

\begin{definition}\label{Def: Bridge}
Let $K \subset S^3$ be a knot. A \defn{bridge sphere} $P \subset S^3$ for $K$ is a 2-sphere transverse to $K$ and separating $S^3$ into two 3-balls, $B_1$ and $B_2$, such that $(B_1, K \cap B_1)$ and $(B_2, K \cap B_2)$ are both trivial tangles. We consider the points $K \cap P$ to be punctures on $P$ (obfuscating the difference between $P$ and $P \setminus K$). We also say that
$(K,P)$ is a \defn{bridge position} for $K$.

We arbitrarily choose one of $B_1$ or $B_2$ to be called the side ``above'' $P$ and the other to be the side ``below'' $P$. Two bridge positions $(K,P)$ and $(J,Q)$ are \defn{isotopically equivalent} if there is an isotopy of $S^3$ which takes $K$ to $J$ and $P$ to $Q$. The \defn{bridge number} of $(K,P)$ is $b(K,P) = |P \cap K|/2$. The \defn{bridge number} of $K$ is
\[
\b(K) = \min_P \{ b(K,P) : P \text{ is a bridge sphere for } K\}.
\]
A bridge position $(K,P)$ with $b(K,P) = \b(K)$ is said to be a \defn{minimal bridge position} with $P$ a \defn{minimal bridge surface}.
\end{definition}

If our bridge position $(K,P)$ arises from a plat diagram of $K$, we note that $b(K,P)$ is equal to both the number of maxima of $K$ and the number of minima of $K$. The bridge number $\b(K)$ of $K$ is an important and useful knot invariant. In particular, if $J$ and $K$ are knots representing the same knot type, then $\b(J) = \b(K)$.

The following elementary observations may be helpful for the reader who has not previously encountered bridge spheres for knots:
\begin{lemma}[Elementary Observations]
The following are true:
\begin{itemize}
\item If $K$ is the plat closure of a braid on $2b$ strands, then $\b(K) \leq b$.
\item If $K$ has a planar diagram with $c$ crossings, then $\b(K) \leq c$.
\item $K$ is the unknot if and only if $\b(K) = 1$.
\end{itemize}
\end{lemma}

One of the reasons we we need to define $\b(K)$ as a minimum, is that if a knot has a bridge position $(K,P)$ then we can create another bridge position $(K,Q)$ with $b(K,Q) = b(K,P) + 1$. We isotope $P$, outside a neighborhood of the punctures, to a sphere $Q$ so that a small disc of $P$ passes through $K$ in such a way that the strands of $Q \setminus K$ are still unknotted, as in Figure \ref{Perturb}. We say that $(K,Q)$ (and any bridge position isotopically equivalent to $(K,Q)$) is a \defn{perturbation} of $(K,P)$.

\begin{figure}[h!]
\centering
\labellist \small\hair 2pt
\pinlabel $P$ [r] at 2 306
\pinlabel $K$ [b] at 112 362
\pinlabel $Q$ [r] at 363 311
\pinlabel $K$ [b] at 472 365
\pinlabel $P$ [r] at 9 106
\pinlabel $K$ [b] at 119 162
\pinlabel $Q$ [r] at 363 111
\pinlabel $K$ [b] at 472 187
\endlabellist
\includegraphics[scale = .5]{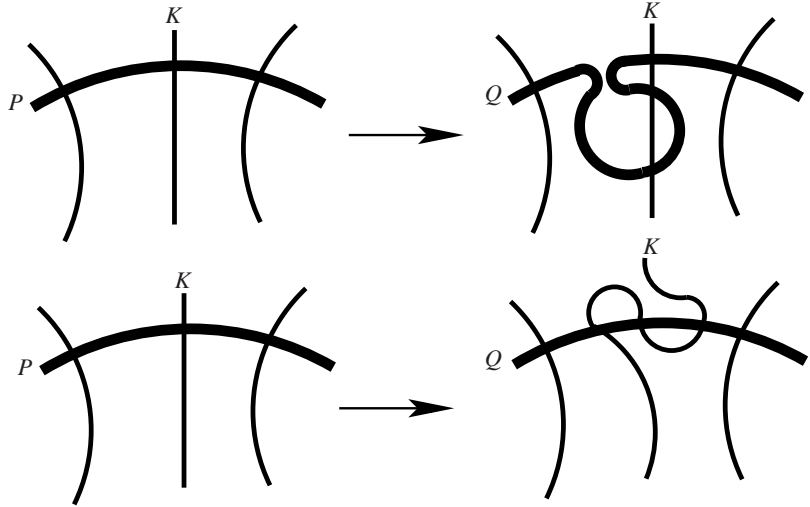}
\caption{The top row shows the perturbation of a bridge sphere $P$. If we fix $P$ and isotope the knot we have the result shown in the bottom row.}
\label{Perturb}
\end{figure}

Although there are evidently many ways of perturbing a bridge sphere $P$, it turns out that any two perturbations of $(K,P)$ are isotopically equivalent.

Given a diagram for $K$ which is a plat closure, it is easy, as we remarked above, to find an upper bound for $\b(K)$. Finding a lower bound for $\b(K)$ is not nearly as easy. For instance, there are knots \cite{IS} with unperturbed non-minimal bridge spheres. The introduction of another invariant, the bridge distance of a knot, will provide us the needed lower bound.

\subsection{The curve complex and bridge distance}

To define the bridge distance of a knot, we make a temporary detour from knot theory and turn to one of the most important constructions in the study of surfaces: the curve complex. For basic results concerning the curve complex, we refer to the marvelous text by Farb and Margalit \cite{FM}.

Let $P$ be a surface (possibly with punctures). A curve $\gamma \subset P$ is \defn{essential} if it does not bound a disc in $P$ containing no punctures or exactly one puncture and does not cobound an annulus with a component of $\boundary P$.  Two essential arcs or curves $\gamma$ and $\gamma'$ in $P$ are \defn{in the same isotopy class} if there is an ambient isotopy of the surface $P$  which takes $\gamma$ to $\gamma'$. If $\gamma$ and $\gamma'$ are arcs, we require that the isotopy be a proper isotopy. We let $[\gamma]$ denote the isotopy class of the curve $\gamma$.

Let $\mc{C}(P)$ (the \defn{curve complex} of $P$) be the graph whose vertex set is the set of isotopy classes of essential curves on $P$ and whose edge set consists of pairs $([\alpha], [\beta])$ of distinct isotopy classes such that $i(\alpha, \beta) = 0$. If $P$ is a sphere with three or fewer punctures, then $\mc{C}(P)$ is empty. If $P$ is a four-punctured sphere or a torus with no punctures, then $\mc{C}(P)$ has countably many vertices and no edges. Otherwise, $\mc{C}(P)$ is infinite and connected (as shown by Lickorish \cite[Theorem 4.3]{FM}). Henceforth, we only consider such surfaces. We turn $\mc{C}(P)$ into a metric space with metric $d$ by declaring each edge to have length one. Thus, if $\gamma$ and $\gamma'$ are essential curves in $P$ and if $d([\gamma], [\gamma']) =  k$, then there is a list of essential curves
\[
\gamma = \gamma_0, \gamma_1, \hdots, \gamma_k = \gamma'
\]
such that for each $i \in \{1, \hdots, k\}$ the curves $\gamma_{i-1}$ and $\gamma_i$ can be isotoped to be disjoint. Figure \ref{ShortPath} shows a path of length 2 in the curve complex of a 6-punctured sphere.

\begin{figure}[h!]
\centering
\labellist \small\hair 2pt
\pinlabel $\gamma_0$ [br] at 11 56
\pinlabel $\gamma_1$ [t] at 116 57
\pinlabel $\gamma_2$ [bl] at 182 54
\endlabellist
\includegraphics[scale = .7]{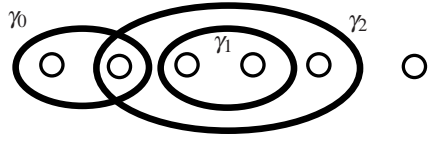}
\caption{The sequence $[\gamma_0], [\gamma_1], [\gamma_2]$ is a path of length 2 in the curve complex of a 6-punctured sphere.}
\label{ShortPath}
\end{figure}

More generally, if $A$ and $B$ are subsets of the vertex set of $\mc{C}(P)$ we define $d(A,B)$ to be the infimum of $d(a,b)$ over all vertices $a \in A$ and $b \in B$.  If $A \cap B \neq \nil$, then $d(A, B) = 0$. Otherwise, if $A \cap B = \nil$, the distance $d(A,B)$ is a natural number. We also remark that any homeomorphism $h \co P \to P$ induces a graph isomorphism $h_*\co \mc{C}(P) \to \mc{C}(P)$. Indeed, the curve complex is one of the principal tools for studying surface automorphisms. 

If $(B,\kappa)$ is a trivial tangle with $P = \boundary B \setminus \kappa$, we will be especially interested in the \defn{disc set} $\mc{D} \subset \mc{C}(P)$. This is the subset of the vertices of $\mc{C}(P)$ which can be represented by curves in $P$ bounding properly embedded discs in $B \setminus \kappa$. Note that if $(B, \kappa)$ is a trivial tangle and if $h \co (B, \kappa) \to (B, \kappa)$ is a homeomorphism of pairs, then the restriction of $h$ to $P$ induces a graph isomorphism of $\mc{C}(P)$ which takes $\mc{D}$ to itself. If $(K,P)$ is a bridge position, there are disc sets for the trivial tangles on either side of $P$. We denote the disc set for the tangle above $P$ by $\A(K,P)$ (or just $\A$ if the context makes the bridge position clear) and the disc set for the tangle below $P$ by $\B(K,P)$ (or just $\B$). The sets $\A$ and $\B$ are the \defn{disc sets} for the bridge position $(K,P)$.  

 We come now to the central tool of this paper:
 \begin{definition}
 The distance of a bridge position $(K,P)$ with $b(K,P) \geq 3$ is
 \[
 d(K,P) = d(\A,\B),
 \]
 where the distance on the right is the distance in the curve complex for $P$ between the disc sets for $(K,P)$.
 \end{definition}

\begin{remark}
If $(K, P)$ and $(J,Q)$ are isotopically equivalent bridge positions by an isotopy $h\co S^3 \to S^3$, then $h$ induces a graph isomorphism $\mc{C}(P) \to \mc{C}(Q)$. If, additionally, $h$ takes  the side above $P$ to the side above $Q$ (and, of course, the side below $P$ to the side below $Q$), then the induced graph isomorphism also takes $\A(K,P)$ to $\A(J,Q)$ and $\B(K,P)$ to $\B(J,Q)$. In particular, if $(K,P)$ and $(J,Q)$ are isotopically equivalent, then $d(K,P) = d(J,Q)$.
 \end{remark}

 One of the central theorems concerning bridge distance is provided by Tomova \cite[Theorem 10.3]{Tomova}.
 \begin{theorem}[Bridge number bounds distance]\label{Tomova}
 Suppose that $(K,P)$ and $(K, Q)$ are bridge positions for a knot $K \subset S^3$ such that  $(K,Q)$ is not obtained from $(K,P)$ by a (possibly empty) sequence of perturbations. Then
 \[
 d(K,P) \leq 2b(K,Q).
 \]
 \end{theorem}
 This has the helpful corollary:
 \begin{corollary}[see {\cite[Corollary 10.6]{Tomova}}]\label{Cor: Tomova}
 Suppose that $(K,P)$ is a bridge position and that
 \[
 d(K,P) > 2b(K,P),
 \]
 then $\b(K) = b(K,P)$ and $(K,P)$ is the unique minimal bridge position for $K$.
 \end{corollary}
 \begin{proof}
We prove the contrapositive. Suppose that $(K,Q)$ is a bridge position for $K$ with either $b(K,Q) < b(K,P)$ or with $b(K,Q) = b(K,P)$ but so that $(K,Q)$ and $(K,P)$ are not isotopically equivalent. Since perturbation only increases bridge number, $(K,Q)$ is not obtained from $(K,P)$ by a sequence of perturbations. Hence, by Theorem \ref{Tomova}, $d(K,P) \leq 2b(K,Q) \leq 2b(K,P)$.
 \end{proof}

Similarly, it follows easily from Theorem \ref{Tomova}, that for a knot $K \subset S^3$, the set \[\{d(K,P) : (K,P) \text{ is a bridge position }\}\] is finite. Thus we can finally define our knot invariant:
\begin{definition}\label{Def: bridge distance}
The \defn{bridge distance} of a knot $K \subset S^3$ with $\b(K) \geq 3$, is
\[
\d(K) = \max_P \{d(K,P) : (K,P) \text{ is a bridge position and } b(K,P) = \b(K)\}.
\]
\end{definition}

We can also obtain a knot invariant by replacing the maximum with a minimum. We choose the maximum as the existence of a high distance bridge sphere for a knot has strong implications for the topological properties of the knot (see, for instance \cite{MHL1}.) For a given knot it can be quite difficult to compute the bridge distance. It is known, however, that knots of arbitrarily high bridge distance (for fixed bridge number) exist \cite{BTY, IS}. Johnson and Moriah \cite{JM} recently constructed a very large class of knots of arbitrary bridge number for which the bridge distances not only go to infinity but can be calculated explicitly from a diagram.

\section{Paths emanating from the Disc Set}

The main challenge in proving Theorem \ref{Main Theorem} is to find an appropriate twisting curve to effect the crossing change. We will find such a curve in a bridge sphere for the knot by finding particular curves far away (in the curve complex) from the disc sets of the knot.

We begin by improving Lemma \ref{Pair Homeo}.

\begin{lemma}\label{Disjt Homeo}
Let $(B,\kappa)$ and $(B,\kappa')$ be two trivial tangles with $\boundary \kappa = \boundary\kappa'$. Let $P = \boundary B \setminus\kappa = \boundary B \setminus \kappa'$. Let $\gamma$ (respectively $\gamma'$) be an essential closed curve in $P$ bounding a disc in $B\setminus\kappa$ (respectively $B\setminus\kappa'$). Then there is a homeomorphism of pairs $h\co (B,\kappa) \to (B,\kappa')$ taking $\gamma$ to a curve disjoint from $\gamma'$.
\end{lemma}

The proof of this lemma is classic cut-and-paste 3-dimensional topology.

\begin{proof}
A \defn{bridge disc} for the trivial tangle$(B,\kappa)$ is a disc embedded in $B$, with interior disjoint from $\kappa$ and whose boundary is the union of a component of $\kappa$ with an arc in $\boundary B$. Since the arcs $\kappa$ can be isotoped to lie in $\boundary B$, there is a collection $\Delta$ of pairwise disjoint bridge discs such that every component of $\kappa$ forms part of the boundary of one disc in $\Delta$. Letting $D \subset B \setminus \kappa$ be the disc bounded by $\gamma$,  out of all such collections of bridge discs, choose $\Delta$ to minimize $|\Delta \cap D|$. We claim that $|\Delta \cap D| = 0$.

If not, $\Delta \cap D$ is the union of finitely many arcs and circles. Suppose that $\zeta \subset \Delta \cap D$ is a circle, chosen to be innermost on $D$. That is, $\zeta$ is the boundary of a disc $D' \subset D$ with interior disjoint from $\Delta$. The circle $\zeta$ also bounds a disc $E' \subset \Delta$ (whose interior may not be disjoint from $D$). Let $\Delta'$ be the result of swapping the disc $E'$ for the disc $D'$ in $\Delta$, as in Figure \ref{InnermostDisc} below. Then after an isotopy to smooth corners and make $\Delta'$ and $D$ transverse, we have $|\Delta' \cap D| < |\Delta \cap D|$, a contradiction. In a similar way, we can also show that $\Delta \cap D$ has no arcs of intersection. Hence, $\Delta \cap D = \nil$.

\begin{figure}[h!]
\centering
\labellist \small\hair 2pt
\pinlabel $\Delta$  at 317 372
\pinlabel $D$ [bl] at 535 388
\pinlabel $D'$ at 384 310
\pinlabel $K$ [b] at 228 400
\pinlabel $\Delta'$ at 341 121
\pinlabel $D$ [bl] at 543 146
\endlabellist
\includegraphics[scale = .5]{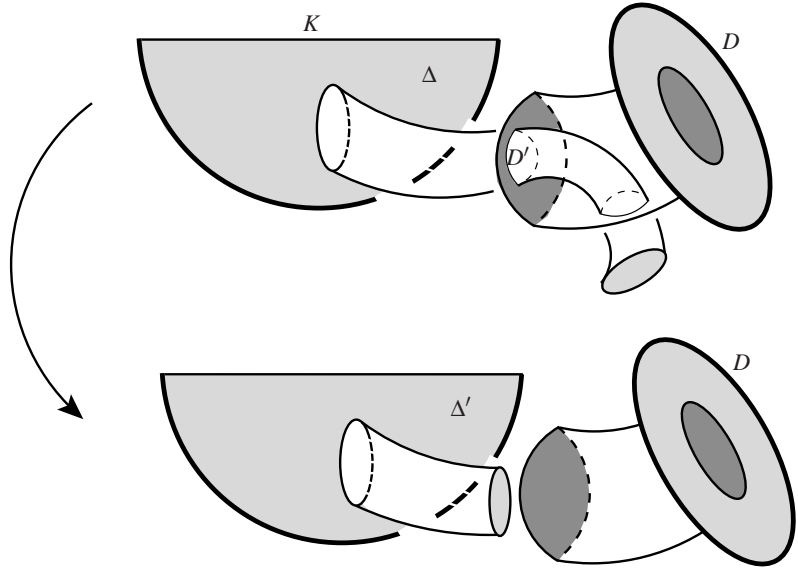}
\caption{A classic innermost disc argument: Doing a disc swap reduces the number of components of intersection between two discs.}
\label{InnermostDisc}
\end{figure}

Since the bridge discs $\Delta$ are disjoint from the disc $D$, cutting $(B,\kappa)$ open along the disc $D$, creates two new trivial tangles $(B_1, \kappa_1)$ and $(B_2, \kappa_2)$. We have $|\kappa_1| + |\kappa_2| = |\kappa|$. Without loss of generality, we may assume that $|\kappa_1| \leq |\kappa_2|$. The curve $\gamma'$  bounds a disc $D' \subset B$ which is disjoint from $\kappa'$. The curve $\gamma'$ also bounds two discs $E_1, E_2 \subset \boundary B$. Cutting $(B, \kappa')$ open along $D'$ produces two trivial tangles. If one of those trivial tangles contains $|\kappa_1|$ arcs and the other $|\kappa_2|$ arcs we let $\gamma'' = \gamma'$ and let $(B'_1, \kappa'_1)$ and $(B'_2, \kappa'_2)$ be the two trivial tangles, respectively. Otherwise, in the boundary of one of those trivial tangles, we can find an essential simple closed curve $\gamma'' \subset P$ which bounds a disc $D'' \subset B \setminus \kappa'$ which is disjoint from $D'$ and so that when we cut $(B, \kappa')$ open along $D''$ we get trivial tangles $(B'_1, \kappa'_1)$ and $(B'_2, \kappa'_2)$ with $|\kappa'_1| = |\kappa_1|$ and $|\kappa'_2| = |\kappa_2|$. Thus, by Lemma \ref{Pair Homeo}, there are homeomorphisms
\[\begin{array}{l}
h_1 \co (B_1, \kappa_1) \to (B'_1, \kappa'_1) \\
h_2 \co (B_2, \kappa_2) \to (B'_2, \kappa'_2).
\end{array}\]

By composing with isotopies of $B'_1$ and $B'_2$ supported outside neighborhoods of $\kappa_1$ and $\kappa_2$,  we may arrange also that the homeomorphisms take the remnant of the disc $D$ to the remnants of the disc $D''$.  The union $h_1 \cup h_2$ may then be extended to a homeomorphism $(B, \kappa) \to (B,\kappa')$ taking $\gamma$ to $\gamma''$.
\end{proof}

Blair, Tomova, and Yoshizawa showed that there exist curves in $\mc{C}(P)$ arbitrarily far from a single disc set. We use this result, together with the previous lemma, to construct curves that are far away from the disc set of a trivial tangle and which have a specified curve in the disc set either as the closest vertex in the disc set or adjacent to the closest vertex. This will be our starting point for finding a twisting curve.

\begin{lemma}\label{convexity}
Let $(B, \kappa)$ be a trivial tangle with $P = \boundary B \setminus\kappa$. Let $\mc{D} \subset \mc{C}(P)$ be the disc set. Choose $[\delta] \in \mc{D}$. Then for every $N \in \N$, there exists $\gamma \subset P$ such that $d([\gamma],\mc{D}) = N$ and $d([\gamma], [\delta]) \leq N + 1$.
\end{lemma}

Figure \ref{Fig: convexity} shows a schematic picture.

\begin{figure}[h!]
\centering
\labellist \small\hair 2pt
\pinlabel $[\gamma]$ [l] at 317 272
\pinlabel $\mc{D}$ at 186 62
\pinlabel $[\delta]$ [br] at 246 164
\pinlabel $N$ [l] at 302 218
\endlabellist
\includegraphics[scale = .5]{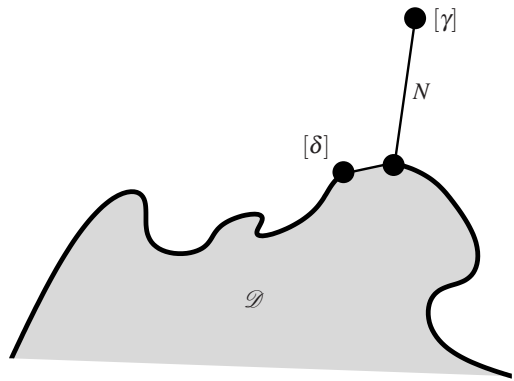}
\caption{Finding the curve $\gamma$. The shaded region represents the disc set for $(B, \kappa)$; each point in the disc set is an essential simple closed curve in $P\setminus \kappa$ which bounds a disc in $B \setminus \kappa$. The line segments represent paths in $\mc{C}(P)$.}
\label{Fig: convexity}
\end{figure}

\begin{proof}
By \cite[Cor. 4.10]{BTY}, there exists $[\ell'] \in \mc{C}(P)$ such that $d([\ell'], \mc{D}) \geq N$. Let $\alpha$ be a path of length $d([\ell'], \mc{D})$ in $\mc{C}(P)$ from $\mc{D}$ to $[\ell']$. There is a vertex $[\ell] \in \alpha$ such that $d([\ell], \mc{D}) = N$. Choose $[\delta'] \in \mc{D}$ to be a vertex such that $d([\delta'], [\ell]) = N$. By Lemma \ref{Disjt Homeo}, there is a homeomorphism $h\co (B,\kappa) \to (B,\kappa)$ such that for the induced map $h_*\co \mc{C}(P) \to \mc{C}(P)$ we have $d(h_*[\delta'], [\delta]) \leq 1$. Note also that $h_*(\mc{D}) = \mc{D}$. Let $[\gamma] = h_*[\ell]$. Hence,
\[
d([\gamma], \mc{D}) = d([\ell],\mc{D}) = N
\]
and, by the triangle inequality,
\[\begin{array}{rcl}
d([\gamma],[\delta]) &\leq& d([\gamma], h_*[\delta']) + d(h_*[\delta'], [\delta]) \\
&\leq & d([\ell], [\delta']) + 1 \\
&=& N + 1,
\end{array}
\]
as desired.
\end{proof}

We apply this to bridge spheres to produce curves arbitrarily far from two disc sets $\A$ and $\B$.

\begin{proposition}[Finding the First Curve]\label{FFC}
Fix $N \in \N$. Let $(K,P)$ be a bridge position for a knot $K$. Then there is a curve $\gamma \subset P$ such that $d([\gamma],\A \cup \B) \geq N$. Furthermore, we may choose $\gamma$ so that $\gamma$ bounds a twice-punctured disc in $P$.
\end{proposition}

Figure \ref{Fig: twoDiscSets} shows a schematic picture.

\begin{figure}[h!]
\centering
\labellist \small\hair 2pt
\pinlabel $[\gamma]$ [l] at 407 308
\pinlabel $\B$ at 179 95
\pinlabel $\A$  at 508 59
\pinlabel $N$ [br] at 341 255
\endlabellist
\includegraphics[scale = .45]{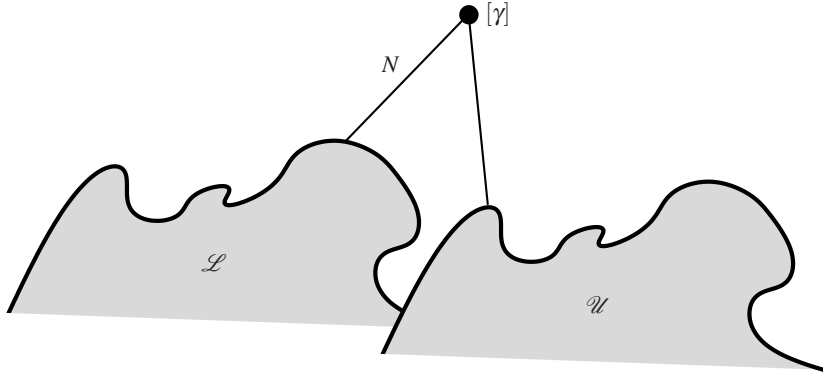}
\caption{We can find a curve $\gamma$ arbitrarily far from two disc sets.}
\label{Fig: twoDiscSets}
\end{figure}

\begin{proof}
First, suppose that we find a curve $\gamma'$ such that $d([\gamma'],\A \cup \B) \geq N + 1$. If $\gamma' \subset P$ bounds a twice-punctured disc, then set $\gamma = \gamma'$ and we are done. If it does not, $\gamma'$ cuts $P$ into two discs, each with at least 3 punctures. Let $\gamma \subset P$ be a curve in one of them bounding a twice-punctured disc in $P$. Then $d([\gamma],[\gamma']) = 1$ and so $d([\gamma], \A \cup \B) \geq N$, as desired. Thus, it suffices to find a curve $\gamma'$ such that $d([\gamma'],\A\cup \B) \geq N + 1$.

Figure \ref{Fig: geometry} indicates the vertices of $\mc{C}(P)$ appearing in the construction of $\gamma'$. By Lemma \ref{convexity}, there is $[\ell] \in \mc{C}(P)$ such that $d([\ell],\A) = 3(N+1)$.  If $d([\ell],\B) \geq N+1$, then by the definition of distance, we are done (after letting $\gamma' = \ell$). Assume, therefore, that $d([\ell],\B) < N+1$. Let $[\delta] \in \B$ be a vertex such that $d([\delta],[\ell]) = d([\ell],\B)$.  Observe that $d([\delta],\A) \geq 2(N+1) + 1$.

Applying Lemma \ref{convexity} again, we may find a vertex $[\gamma'] \in \mc{C}(P)$ such that $d([\gamma'],\B) = N+1$ and $d([\gamma'], [\delta]) \in \{N+1, N + 2\}$. Consequently,
\[\begin{array}{rcl}
d([\gamma'],\A) + (N + 2) &\geq& d([\gamma'],\A) + d([\delta],[\gamma']) \\
&\geq& d([\delta],\A ) \\
&\geq& 2(N+1) + 1.
\end{array}
\]
Hence, $d([\gamma'],\A \cup \B) \geq N+1$ which implies $d([\gamma'],\A \cup \B) = N+1$.
\end{proof}

\begin{figure}[h!]
\centering
\labellist \small\hair 2pt
\pinlabel $[\gamma']$ [r] at 141 311
\pinlabel $[\ell]$ [r] at 236 319
\pinlabel $\B$ at 186 98
\pinlabel $\A$ at 512 59
\pinlabel $3(N+1)$ [bl] at 365 189
\pinlabel $[\delta]$ [b] at 299 207
\pinlabel $N+1$ [tr] at 189 259
\endlabellist
\includegraphics[scale = .45]{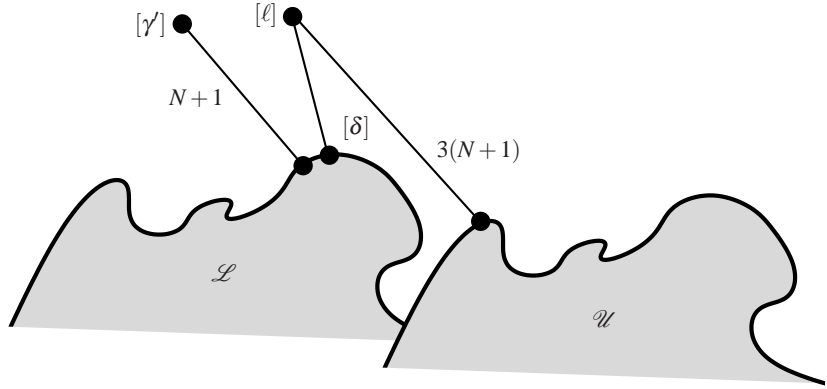}
\caption{Finding the curve $\gamma'$}
\label{Fig: geometry}
\end{figure}

\section{Subsurface Projections}

A separating curve in a surface $P$ cuts the surface into two components. We will also need a version of the curve complex for those components, which can be easily related to the curve complex for $P$.

\begin{definition}
Suppose that $F$ is a surface. An arc $\gamma \subset F$ is \defn{essential} if there does not exist an arc $\alpha \subset \boundary F$ such that $\alpha \cup \gamma$ bounds a disc in $F$. The \defn{arc-and-curve complex} for $F$ is a graph $\mc{AC}(F)$ with vertex set the set of isotopy classes of essential arcs and loops in $F$ and with an edge between two distinct vertices if they have representatives which are disjoint. As we did with the curve complex, we make $\mc{AC}(F)$ into a metric space with metric $d_F$ by declaring each edge to have length one. As usual, if $A$ and $B$ are subsets of the vertex set of $\mc{AC}(F)$, we let $d_F(A,B) = \min\{ d_F(a,b): a \in A, b \in B\}$.
\end{definition}

If $P$ is a surface and if $F \subset P$ is also a surface, any simple closed curve $\gamma$ in $P$ which is transverse to $\boundary F$ is either contained in $F$, or cuts through $F$ as a collection of arcs (as in Figure \ref{Fig: Projection}), or is disjoint from $F$.

\begin{figure}[h!]
\centering
\labellist \small\hair 2pt
\pinlabel $P$ [bl] at 7 7
\pinlabel $F$ [l] at 239 89
\pinlabel $\gamma$ [b] at 579 125
\endlabellist
\includegraphics[scale = 0.5]{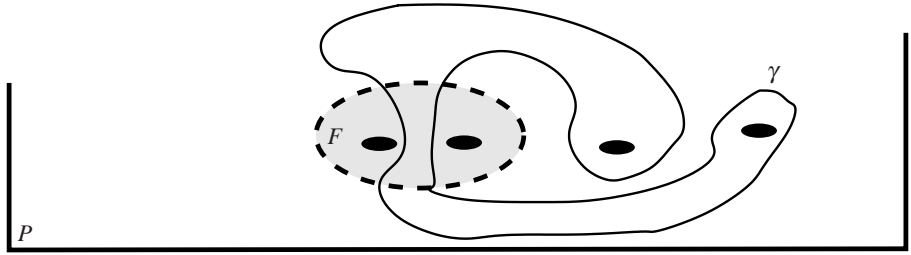}
\caption{The surface $P$ is a sphere with at least 4 punctures and the subsurface $F$ is a disc with at least 2 punctures. The curve $\gamma$ intersects $F$ in two arcs. }
\label{Fig: Projection}
\end{figure}

This observation allows us to define the so-called \defn{projection map} from $\mc{C}(P)$ to $\mc{AC}(F)$. The map is only useful for certain kinds of subsurfaces and we want to apply it to subsets of the vertex set of $\mc{C}(P)$ as we describe in the next definition.

\begin{definition}
Suppose that $P$ is a punctured surface and that $F \subset P$ is a connected subsurface with $\boundary F$ disjoint from the punctures. The subsurface $F$ is an \defn{essential subsurface} of $P$ if $\boundary F \neq \nil$ and if each component of $\boundary F$ is an essential curve in $P$. If $F \subset P$ is an essential subsurface and if $A$ is a set of vertices of $\mc{C}(P)$, we let $\pi_F(A)$ denote the set of vertices of $\mc{AC}(F)$ such that for each $[v] \in \pi_F(A)$ there is a vertex $[w] \in A$ and a simple closed curve $\alpha \in [w]$ such that $\alpha$ intersects each component of $\boundary F$ minimally and some component of $\alpha \cap F$ is a representative of $[v]$. The set $\pi_F(A)$ is called the \defn{subsurface projection} of $A$ onto $F$.
\end{definition}

\begin{remark}
In the example from Figure \ref{Fig: Projection}, $\gamma$ intersects $\boundary F$ minimally. The two components of $\gamma \cap F$ are parallel arcs and so represent the same vertex of $\mc{AC}(F)$ (which is, in fact, the only vertex of $\mc{AC}(F)$). Thus, $\pi_F([\gamma])$ is a single vertex in $\mc{AC}(F)$. If the curve $\gamma$ were disjoint from $F$, then $\pi_F([\gamma])$ would be empty.
\end{remark}

We use subsurface projections to define the ``subsurface distance'' between isotopy classes of curves in $P$.

\begin{definition}
Suppose that $P$ is a punctured surface and that $F \subset P$ is an essential subsurface. Suppose that $A$ and $B$ are subsets of vertices of $\mc{C}(P)$. If either $\pi_F(A)$ or $\pi_F(B)$ is empty, we define $d_F(A,B) = \infty$. Otherwise, we let $d_F(A,B) = d_F(\pi_F(A), \pi_F(B))$. Similarly, we let $\diam_F(A)$ denote the diameter of the set $\pi_F(A)$. That is, the supremum (which may be $\pm\infty$) of $\{d_F(v,w) : v, w \in \pi_F(A)\}$.
\end{definition}

The following lemma shows a very useful relationship between distance in $\mc{C}(P)$ and subsurface distance. We will use it many times in what follows.

\begin{lemma}\label{Minimal paths}
Suppose that $P$ is a sphere with at least 6 punctures. Let $F \subset P$ be an essential subsurface and let $A,B \subset \mc{C}(P)$. Then
one of the following occurs:
\begin{itemize}
\item $d_F(A,B) \leq d(A,B)$
\item Every minimal length path in $\mc{C}(P)$ from $A$ to $B$ includes the isotopy class of a simple closed curve disjoint from $F$.
\end{itemize}
\end{lemma}

\begin{proof} If either $\pi_F(A)$ or $\pi_F(B)$ is empty, the result is easily seen to be true, so we may assume both projections are non-empty. Choose $a \in A$ and $b \in B$ so that $d(a,b) = d(A,B)$. Let
\[
a = [\gamma_0], [\gamma_1], [\gamma_2], \hdots, [\gamma_n] = b
\]
be the vertices of a minimal length path from $a$ to $b$ in $\mc{C}(P)$. We choose our notation so that each $\gamma_i$ (for $i \in \{0, \hdots, n\}$) intersects $F$ minimally and for all $i \neq j$ the curves $\gamma_i$ and $\gamma_j$ also intersect minimally. In particular, for all $i \in \{0, \hdots, n-1\}$, $\gamma_i \cap \gamma_{i+1} = \nil$. We will show that either $d_F(A,B) \leq d(A,B)$ or that some $\gamma_i$ is disjoint from $F$.

Assume that for all $i$, the curve $\gamma_i$ intersects $F$. This implies that for all $i \in \{0, \hdots, n\}$ we have $\pi_F([\gamma_i]) \neq \nil$.  We will show that $d_F(A,B) \leq d(A,B)$.

For each $i \in \{0, \hdots, n\}$, let $\rho_i$ be a component of $\gamma_i \cap F$. Since $\gamma_i$ and $F$ intersect minimally, $\rho_i$ is an essential arc or simple closed curve in $F$.  Let $[v_i]$ be the isotopy class of $\rho_i$ in $F$, relative to its endpoints. Since, for each $i \in \{0, \hdots, n-1\}$, $\gamma_i \cap \gamma_{i+1} = \nil$, the arcs $\rho_i$ and $\rho_{i+1}$ are also disjoint. Hence,
\[
[v_0], \hdots, [v_n]
\]
are the vertices of a path in $\mc{AC}(F)$ from $\pi_F(a)$ to $\pi_F(b)$. Hence,
\[
d_F(A,B)\leq d_F(a,b) \leq d(a,b) = d(A,B).
\]
\end{proof}

We can now use this result to construct paths which must pass through a certain given curve.

\begin{lemma}\label{First step}
Suppose that $P$ is a sphere with at least 6 punctures. Let $c \subset P$ be an essential simple closed curve bounding a twice-punctured disc $G$ in $P$. Let $F = P\setminus G$. Let $\mc{B} \subset \mc{C}(P)$ be a subset of the vertices such that $\pi_F(\mc{B})$ is non-empty and bounded. Let $n \in \N$. Then there exists a simple closed curve $\ell \subset P$ such that all the following hold:
\begin{itemize}
\item $\ell \subset F$ bounds a twice-punctured disc in $F$.
\item $d([\ell],\mc{B}) = d([c],\mc{B}) + 1$.
\item For every vertex $[b] \in \mc{B}$, every minimal length path in $\mc{C}(P)$ from $[\ell]$ to $[b]$ passes through $[c]$.
\item $d_F([\ell], \mc{B}) \geq n$.
\end{itemize}
\end{lemma}
\begin{proof}
Begin by observing that any essential simple closed curve in $P$ which can be isotoped to lie in $G$ is actually isotopic to $c$, as $G$ is a twice-punctured disc.

Let $x = \max(n, d([c],\mc{B}))$. Since $P$ is a sphere with at least 6 punctures,  $F$ is a disc with at least 4 punctures. Since $\mc{AC}(F)$ is infinite and connected but $\pi_F(\mc{B})$ is bounded, there is an essential arc or simple closed curve $\alpha \subset F$ such that $d_F([\alpha], \mc{B}) \geq x + 3$. Since $F$ has at least 4 punctures, there is a simple closed curve $\ell \subset F\setminus \alpha$ (possibly isotopic to $\alpha$) such that $\ell$ bounds a twice-punctured disc in $P$.  Since $[\alpha]$ is at least $x + 3$ from $\pi_F(\mc{B})$ in $\mc{AC}(F)$, the vertex $[\ell]$ is at least $x+2$ from $\pi_F(\mc{B})$ in $\mc{AC}(F)$.

Let $[b] \in \mc{B}$. We show that every minimal length path from $[\ell]$ to $[b]$ in $\mc{C}(P)$ passes through $[c]$. Suppose not. Then by Lemma \ref{Minimal paths}, $d_F([\ell],[b]) \leq d([\ell],[b])$. Since the curve $\ell$ is disjoint from $c = \boundary G$, we have
\[\begin{array}{rcl}
d([c],\mc{B}) + 1 & \geq & d([\ell],[b]) \\
&\geq& d_F([\ell],[b]) \\
&\geq & d_F([\ell],\mc{B}) \\
&\geq& x + 2 \\
& \geq & d([c],\mc{B}) + 2.
\end{array}
\]
As this is nonsensical, every minimal length path from $[\ell]$ to $[b]$ in $\mc{C}(P)$ must pass through $[c]$.

Now fix $[b] \in \mc{B}$ to be the vertex closest to $[\ell]$. Let
\[
[\ell] = [\alpha_0], [\alpha_1], \hdots, [\alpha_k] = [b]
\]
be a minimal path from $[\ell]$ to $[b]$. By the previous paragraph, $[\alpha_i] = [c]$ for some $i$.

Consequently
\[
[\alpha_i], [\alpha_{i+1}], \hdots, [\alpha_k]
\]
is a path from $[c]$ to $[\mc{B}]$ of length $k - i$. Thus,
\[
d([c],\mc{B}) \leq k - i.
\]
As $\ell$ is not isotopic to $c$, $i \geq 1$, so $d([c],\mc{B}) + 1 \leq k = d([\ell],\mc{B})$. Since we also have $d([\ell],\mc{B}) \leq d([c],\mc{B}) + 1$, we see that $d([\ell],\mc{B}) = d([c],\mc{B}) + 1$, as desired.
\end{proof}

We will apply Lemma \ref{First step} several times. The first time we do so, we take the set $\mc{B}$ to be equal to the disc sets for a bridge sphere. The next lemma confirms that this is possible.

\begin{lemma}\label{diameter}
Let $(K,P)$ be a bridge position. Suppose that $c \subset P$ is an essential curve such that for all $[\delta] \in \A \cup \B$ we have $i(\delta,c) > 2$. Let $F$ be the closure of a component of $P \setminus c$. Then the sets $\pi_F(\A)$ and $\pi_F(\B)$ each have diameter at most 4.
\end{lemma}

\begin{proof}
We prove the lemma for $\A$. The result for $\B$ follows by interchanging $\A$ and $\B$ in the following. If $F$ is a twice-punctured disc, $\mc{AC}(F)$ is a single vertex and the result is trivially true. We assume, therefore, that $F$ is a disc with at least 3 punctures. Fix an essential simple closed curve $\delta \subset P$  such that $[\delta] \in \A$ and so that $|\delta \cap c|$ is minimal (out of all curves representing vertices in $\A$). By hypothesis, $|\delta \cap c| > 2$.

Let $\epsilon$ be any curve in $P$ such that $[\epsilon] \in \A$ and $\epsilon$ intersects $c$ and $\delta$ minimally in its isotopy class. We will show that $d_F([\delta],[\epsilon]) \leq 2$. Since $[\delta]$ is fixed and $[\epsilon]$ is arbitrary, it will follow that $\pi_F(\A)$ has diameter at most 4.

If $d_F([\delta], [\epsilon]) < 2$ we are done, so assume that $d_F([\delta], [\epsilon]) \geq 2$. This implies that no component of $\epsilon \cap F$ can be isotoped to be disjoint from any component of $\delta \cap F$ in $F$. Let $D$ and $E$ be discs on the same side of $P$, with boundaries equal to the curves $\delta$ and $\epsilon$, respectively and with interiors disjoint from $P \cup K$. Out of all such discs, choose $D$ and $E$ to minimize $|D \cap E|$. Since $\epsilon$ and $\delta$ are not disjoint, $D \cap E$ is non-empty. As in the proof of Lemma \ref{Disjt Homeo}, we may conclude that $D \cap E$ consists entirely of arcs of intersection (i.e. no circles).

Let $\alpha \subset D \cap E$ be an outermost arc of intersection cutting off a disc $E' \subset E$ with interior disjoint from $D$. Let $\beta$ be the closure of the arc $\boundary E' \setminus \alpha$. Cutting $D$ open along $\alpha$ and patching in two parallel copies of $E'$ (as in Figure \ref{ChoppingDisc}) we arrive at discs $D'_1$ and $D'_2$. Both $\boundary D'_1$ and $\boundary D'_2$ must be essential curves in $P$, as otherwise we could isotope $E$ to reduce $|D \cap E|$, a contradiction to our original choice of the discs $D$ and $E$. By our original choice of the curve $\delta$, $i(\boundary D'_1, c) \geq i(\delta, c)$ and $i(\boundary D'_2, c) \geq i(\delta, c)$.

Hence, \[|\boundary D'_1 \cap c| + |\boundary D'_2 \cap c| \geq 2|\delta \cap c|.\]

From the cut-and-paste construction, it is evident that
\[\begin{array}{rcl}
|\delta \cap c| + 2 |\beta \cap c| &=& |\boundary D'_1 \cap c| + |\boundary D'_2 \cap c|\\
&>& 4.
\end{array}
\]

Thus, $|\beta \cap c| \geq 2$. Since $\epsilon$ intersected $c$ minimally, $\beta$ cannot be isotoped, relative to its endpoints, to decrease $|\beta \cap c|$. Let $\beta'$ be a component of $\beta \cap F$ having both endpoints on $c$ (i.e. with both endpoints in the interior of $\beta$). Observe that $\epsilon$ can be isotoped to be disjoint from $\beta'$ (since $\beta' \subset \epsilon$). Since the interior of $\beta$ is disjoint from $\delta$, we also have $\beta' \cap \delta = \nil$. Hence, \[d_F([\delta], [\epsilon]) \leq d_F([\delta],[\beta']) + d_F([\beta'],[\epsilon]) \leq 2,\] as desired.
\end{proof}

\begin{figure}
\centering
\labellist \small\hair 2pt
\pinlabel $E$ at 395 431
\pinlabel $D$ [b] at 580 355
\pinlabel $E'$ at 400 332
\endlabellist
\includegraphics[scale=.4]{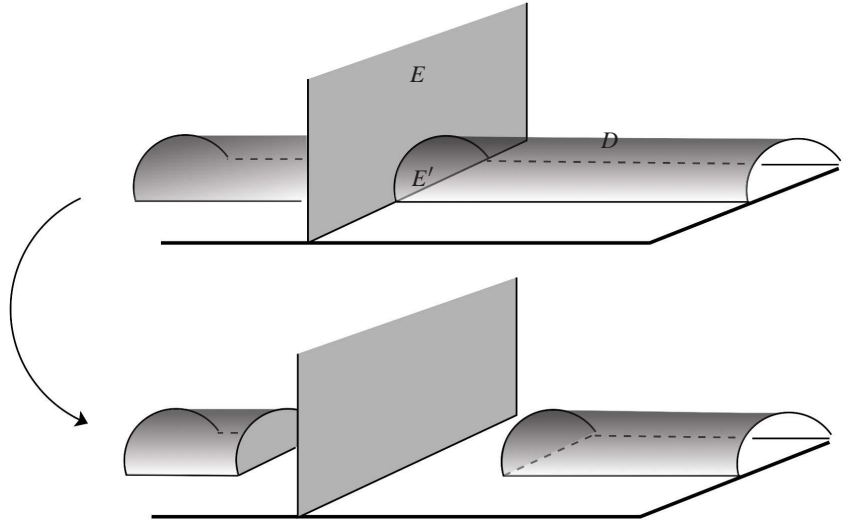}
\caption{Cutting open $D$ and patching in $E'$ creates two new discs.}
\label{ChoppingDisc}
\end{figure}

\section{Dehn twists on surfaces}
Our aim is to start with a given knot $K$, perform a single crossing change, and end up with a knot $K'$ of arbitrarily high distance. Performing a crossing change amounts to twisting a knot along a crossing disc. We will find an appropriate disc by a close examination of the curve complex for a certain bridge sphere for our starting knot $K$.

\begin{definition}
Suppose that $P$ is a surface containing an oriented curve $\gamma$. A \defn{Dehn twist} of $P$ along $\gamma$ is a homeomorphism $\tau \co P \to P$ defined by cutting $P$ open along $\gamma$, twisting by $2\pi$ in the direction of the orientation, and then regluing. This homeomorphism is supported in a small neighborhood of $\gamma$. A Dehn twist around the boundary of a twice-punctured disc which is a subset of a bridge sphere for a knot $K$ results in a knot $K'$ obtained from $K$ by a crossing change with the twice-punctured disc playing the role of the crossing disc.
\end{definition}

\begin{remark}
A curve $\beta$ intersecting $\gamma$ transversally in a point $p$, after the twist $\tau$, will, as it approaches $\gamma$ from the right (where right and left are defined by the orientation of $\gamma$) veer to the right, following parallel to $\gamma$ until it comes back to where it started at which point it will cross $\gamma$ at $p$. We can find a curve isotopic to $\tau(\beta)$ by taking parallel copies of $\gamma$, one for each point in $\beta \cap \gamma$, and then resolve the intersections between those parallel copies and $\beta$, as in Figure \ref{Fig: Dehn Twist}.
\end{remark}

\begin{figure}
\centering
\labellist \small\hair 2pt
\pinlabel $\gamma$ [b] at 8 94
\pinlabel $\beta$ [t] at 65 4
\pinlabel $\tau(\beta)$ [t]  at 339 3
\endlabellist
\includegraphics[scale = .5]{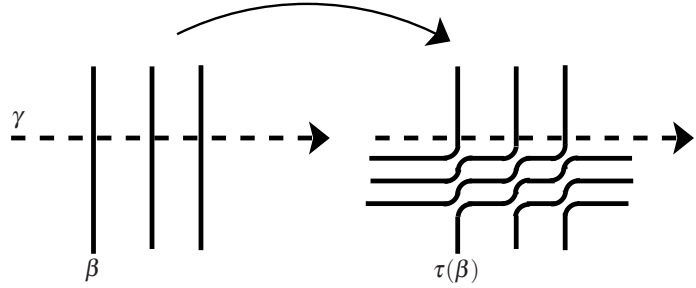}
\caption{The result of a Dehn Twist}
\label{Fig: Dehn Twist}
\end{figure}

\subsection{Distance and Intersection Number}
Suppose that  $\alpha$ and $\beta$ are curves transverse to the oriented simple closed curve $\gamma$ which is disjoint from the points $\alpha \cap \beta$. Letting $\tau$ be the Dehn twist around $\gamma$ we see that
\[
|\alpha \cap \tau(\beta)| = |\alpha \cap \beta| + |\alpha \cap \gamma| |\beta \cap \gamma|.
\]
Equivalently, 
\[
|\alpha \cap \tau(\beta)| - |\alpha \cap \gamma| |\beta \cap \gamma| = |\alpha \cap \beta|.
\]
Unfortunately, we do not know that $\tau(\beta)$ intersects $\alpha$ minimally, even if $\alpha$, $\beta$, and $\gamma$ all pairwise intersect each other minimally. The next lemma, however, gives us the needed control by showing that we can pass to intersection numbers at the cost of turning the previous equality into inequalities. It is proved by using the fact that if two curves do not intersect minimally, then there is a bigon between them. We refer the reader to \cite[Proposition 3.4]{FM} for a proof.

\begin{lemma}[Intersection Lemma]\label{IL}
Suppose that $\alpha$, $\beta$, and $\gamma$ are essential simple closed curves in $P$, with $\gamma$ oriented. Let $\tau\co P \to P$ be a Dehn twist around $\gamma$. Then the intersection numbers satisfy:
\[
|i(\alpha, \tau(\beta)) - i(\alpha,\gamma)i(\beta, \gamma)| \leq i(\alpha,\beta).
\]
\end{lemma}

Finally, note that there is a relationship between the intersection number of two arcs or curves in a surface $F$ and their distance in $\mc{AC}(F)$ (if at least one is an arc) or $\mc{C}(F)$ (if both are curves). The proof of the following lemma is left as an exercise in cut-and-paste topology and counting. The essence of the proof may be found in \cite[Lemma 2.1]{Hempel}.

\begin{lemma}\label{intersections bound distance}
If $\alpha$ and $\beta$ are essential arcs or curves in $F$ such that $i(\alpha,\beta) > 0$ then the distance from $[\alpha]$ to $[\beta]$ in $\mc{AC}(F)$ or $\mc{C}(F)$ is at most
\[
2 + 2\log_2(i(\alpha,\beta)).
\]
\end{lemma}

\begin{remark}\label{rem: intersections}
This bound is not always the best possible. For instance, if $F$ is a sphere with at least 6 punctures and if $\alpha$ and $\beta$ are essential simple closed curves with $3 \leq d([\alpha],[\beta])$, then we must actually have $i(\alpha,\beta) \geq 3$.
\end{remark}

\subsection{Dehn twists and disc sets}
Suppose that $(K,P)$ is a bridge position and that $\gamma$ is an oriented essential curve in the punctured surface $P$ bounding a punctured disc $D$ in $P$. Let $(B,\kappa)$ be the trivial tangle above $P$ and $(B',\kappa')$ the trivial tangle below $P$. The pair $(S^3, K)$ is formed by gluing via a homeomorphism $h\co (\boundary B, \boundary \kappa) \to (\boundary B' , \boundary \kappa')$. Let $\tau$ be a Dehn twist around $\gamma$. The knot $K'$ resulting from twisting $K$ along the disc $D$ is formed by gluing $(B,\kappa)$ to $(B',\kappa')$ via the homeomorphism
\[
\tau \circ h\co(\boundary B , \boundary \kappa) \to (\boundary B' , \boundary \kappa').
\]
Thus, $(K',P)$ is still a bridge position. By identifying $(B',\kappa')$ with the tangle $(B', K' \cap B')$, we can assume that in $\mc{C}(P)$ we have $\A(K',P) = \tau_*\A(K,P)$ and $\B(K',P) = \B(K,P)$.

Finally, observe that if the disc $D$ has exactly two punctures, then $K'$ is obtained from $K$ by a crossing change.

\section{Proof of Main Theorem}
Let $K \subset S^3$ be a knot and suppose that $(b,n) \in \N^2$ with $b \geq \max(3,\b(K))$. We will show that there is a knot $K' \subset S^3$ with $\b(K') = b$ and $\d(K') \geq n$ such that $K'$ and $K$ differ by a single crossing change. The remainder of the section is taken up with the proof.

Perturb a minimal bridge sphere for $K$ a total of $b - \b(K)$ times to produce a bridge position $(K,P)$ with $b(K,P) = b$. We now set about finding a simple closed curve $\ell_3 \subset P$ bounding a twice punctured disc $D \subset P$ such that our knot $K'$ will be obtained from $K$ by twisting along $D$. Let $\A = \A(K,P)$ and $\B = \B(K,P)$ be the disc sets. Let $\mc{B} = \A \cup \B$.

Our strategy is to begin by finding a curve $\ell_1$ far from both disc sets and then carefully moving two steps beyond it to the desired curve $\ell_3$, as in Figure \ref{strategy}. We will then show that performing a Dehn twist along $\ell_3$ moves the disc set $\A$ to a disc set $\A'$ far from $\B$. We will conclude that the new knot $K'$ has the desired properties by appealing to Corollary \ref{Cor: Tomova}.

\begin{figure}[h!]
\centering
\labellist \small\hair 2pt
\pinlabel $[\ell_3]$ [l] at 287 484
\pinlabel $[\ell_2]$ [r] at 271 468
\pinlabel $[\ell_1]$ [l] at 272 442
\pinlabel $\B$ at 157 336
\pinlabel $\A$  at 322 315
\pinlabel $[\ell_3]$ [l] at 266 184
\pinlabel $[\ell_2]$ [r] at 251 164
\pinlabel $[\ell_1]$ [r] at 239 145
\pinlabel $\B$ at 131 33
\pinlabel $\A'$  at 409 132
\endlabellist
\includegraphics[scale = .5]{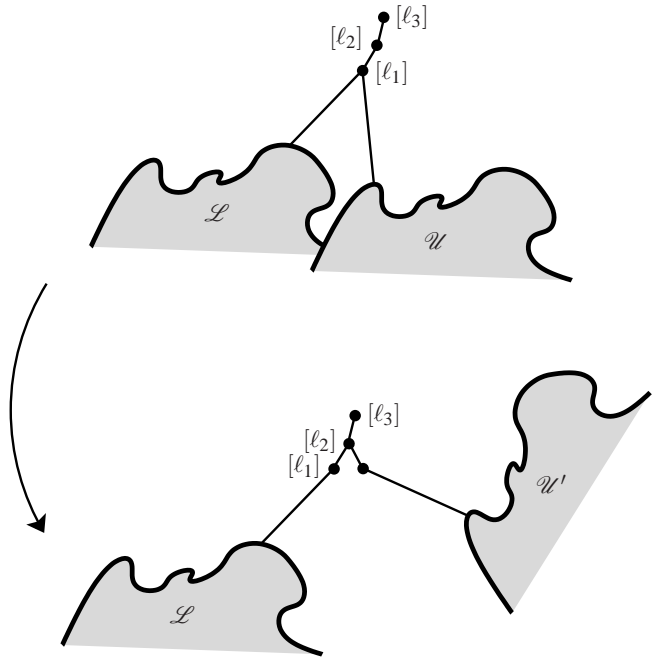}
\caption{The result of a Dehn twist around $\ell_3$ on disc complexes.}
\label{strategy}
\end{figure}

Let $N > \max(n, 2b+1, 3)$.  By Proposition \ref{FFC}, there is a simple closed curve $\ell_1 \subset P$ bounding a twice-punctured disc $G_1 \subset P$ such that $d([\ell_1],\mc{B}) \geq N+8$.  Let $F_1 = P \setminus G_1$. By Lemma \ref{intersections bound distance}, $i(\ell_1, \delta) > 2$ for every $[\delta] \in \mc{B}$. By Lemma \ref{diameter},  the subsurface projections $\pi_{F_1}(\B)$ and $\pi_{F_1}(\A)$ each have diameter at most 4. Thus, $\pi_{F_1}(\mc{B})$ is bounded. By Lemma \ref{First step}, there is a simple closed curve $\ell_2 \subset F_1$, bounding a twice-punctured disc $G_2\subset F_1$ such that
\[
d([\ell_2],\mc{B}) = d([\ell_1],\mc{B}) + 1 \geq N + 9,
\]
every minimal path from $[\ell_2]$ to $\mc{B}$ passes through $[\ell_1]$ and $d_{F_1}([\ell_2], \mc{B}) \geq N+9$. Let $F_2 = P \setminus G_2$.

Pick an arbitrary $\delta \in \A$ and isotope it to intersect $\ell_1$ and $\ell_2$ minimally. Observe that $\delta$ is not disjoint from either $\ell_1$ or $\ell_2$ and that $\pi_{F_2}(\mc{B} \cup \{[\ell_1]\})$ is bounded (appealing again to Lemma \ref{diameter} and Lemma \ref{intersections bound distance}). Let $p = i(\ell_1,\delta) + i(\ell_2, \delta)$.

By Lemma \ref{First step}, there exists a simple closed curve $\ell_3 \subset F_2$ such that the following hold:
\begin{itemize}
\item $\ell_3$ bounds a twice punctured disc $D$ in $F_2$;
\item $\ell_3$ intersects the curves $\ell_1$ and $\delta$ minimally;
\item Every minimal path in $\mc{C}(P)$ from $[\ell_3]$ to $\mc{B}$ passes through $[\ell_1]$ and $[\ell_2]$; and
\item $d_{F_2}([\ell_3], \mc{B} \cup \{[\ell_1]\}) > 2 + 2\log_2(p)$.
\end{itemize}

Orient $\ell_3$. Let $\tau_*$ be the automorphism of $\mc{C}(P)$ corresponding to a Dehn twist $\tau$ around $\ell_3$. Let $K'$ be the resulting knot and let $\A' = \A(K',P) = \tau_*\A(K,P)$ be the disc sets. Observe that $K'$ is obtained from $K$ by a single crossing change and that $d(K',P) = d(\A', \B)$. Thus, as long as $d(\A',\B)$ is at least $N$, by Corollary \ref{Cor: Tomova}, we have $n \leq d(K',P) = d(K')$ and $\b(K')=b$ as desired. The next lemma helps us verify that $d(\A',\B) \geq N$.

\begin{lemma}[Twisting kills subsurface distance]\label{Twisting kills subsurface distance}
The subsurface projection of the disc set $\A'$ onto the subsurface $F_1$ (which has boundary $\ell_1$) is distance at most 1 from the class $[\ell_2]$ in the arc-and-curve complex $\mc{AC}(F_1)$. That is,
\[ d_{F_1}([\ell_2], \A') \leq 1 \]
\end{lemma}
\begin{proof}
To show that $d_{F_1}([\ell_2], \A') \leq 1$, we need only show that there is a vertex $v \in \A'$ and a representative $\gamma$ of $v$ such that $\ell_2$ can be isotoped to be  disjoint from an arc of $\gamma \cap F_1$ and $\gamma$ intersects $\ell_1 = \boundary F_1$ minimally up to isotopy. It turns out that any class $v = [\tau(\delta)]$ will do the trick. The main difficulty in showing this stems from the fact that although $\delta$ and $\ell_3$ intersect $\ell_1$ minimally, the curve $\tau(\delta)$ may not intersect $\ell_1$ minimally. We will use the Intersection Lemma to handle this.

By Lemma \ref{IL},
\begin{equation}\label{Ineq 1}
\begin{array}{rcl}
i(\ell_1, \tau(\delta)) &\geq& i(\ell_1, \ell_3) i(\ell_3, \delta) - i(\ell_1, \delta)\\
&\geq& i(\ell_1, \ell_3) - i(\ell_1, \delta)
\end{array}
\end{equation}
where the last inequality holds since $i(\ell_3, \delta) > 0$.

By Lemma \ref{intersections bound distance} and our choice of $\ell_3$,
\[
i(\ell_1, \ell_3) \geq 2^{(d_{F_2}([\ell_1],[\ell_3])-2)/2}  > i(\ell_2, \delta) + i(\ell_1, \delta).
\]
Observe, also, that since $\ell_3 \cap \ell_2 = \nil$, $i(\ell_2,\delta) = i(\ell_2, \tau(\delta))$. Combining with Inequality \eqref{Ineq 1} we obtain:
\begin{equation}\label{Ineq 2}
i(\ell_1, \tau(\delta)) > i(\ell_2, \tau(\delta)).
\end{equation}

Let $\gamma$ be the result of isotoping $\tau(\delta)$ to minimize the number of intersections with both $\ell_1$ and $\ell_2$. Then $\gamma \in \A'$ and $\gamma$ intersects $\ell_1$ minimally in exactly $i(\ell_1, \tau(\delta))$ points. Since each arc has two endpoints, $\gamma \cap F_1$ has exactly $a = i(\ell_1,\tau(\delta))/2$ arcs. The curve $\ell_2$ separates the surface $F_1$, so each arc component of $\gamma \cap F_1$ must intersect $\ell_2$ in an even number of points. Thus, if  no component of $\gamma \cap F_1$ is disjoint from $\ell_2$ we would have
\[
i(\ell_2, \tau(\delta)) \geq 2a = i(\ell_1, \tau(\delta)).
\]

However, this contradicts Inequality \eqref{Ineq 2}, so at least one component of $\gamma \cap F_1$ is disjoint from $\ell_2$. Since $\gamma$ intersects $F_1$ minimally in its isotopy class, this is enough to show that $d_{F_1}([\ell_2], [\tau(\delta)]) \leq 1$, as desired.
\end{proof}

Let $\alpha_1 \subset P$ be a simple closed curve, isotoped to intersect $F_1$ minimally such that $[\alpha_1] \in \B$ and so that there is a component $\alpha'_1 \subset \alpha_1 \cap F_1$ with
\[
d_{F_1}([\alpha_1'],[\ell_2]) = d_{F_1}(\B, [\ell_2]).
\]
Similarly, let $\alpha_2 \subset P$ and $\beta_1 \subset P$ be simple closed curves, isotoped to intersect $F_1$ minimally, such that $[\alpha_2] \in \B$, $[\beta_1] \in \A'$ and so that there are components $\alpha'_2 \subset (\alpha_2 \cap F_1)$ and $\beta'_1 \subset (\beta_1 \cap F_1)$ with
\[
d_{F_1}([\alpha'_2], [\beta_1']) = d_{F_1}(\A', \B).
\]
Finally, let $\beta_2 \subset P$ be a simple closed curve, isotoped to intersect $F_1$ minimally, such that $[\beta_2] \in \A'$ and so that there is a component $\beta'_2 \subset (\beta_2 \cap F_1)$ with $d_{F_1}([\beta'_2], [\ell_2]) \leq 1$. (Such a curve exists by Lemma \ref{Twisting kills subsurface distance}.)

By Lemma \ref{diameter}, $d_{F_1}([\alpha'_1], [\alpha'_2]) \leq 4$ and $d_{F_1}([\beta'_1], [\beta'_2]) \leq 4$. Thus, by our choice of $\ell_2$ and the triangle inequality we have:
\[\begin{array}{rcl}
N + 9&\leq& d_{F_1}([\alpha'_1], [\ell_2]) \\
&\leq& d_{F_1}([\alpha'_1], [\alpha'_2]) + d_{F_1}([\alpha'_2], [\beta_1']) + d_{F_1}([\beta_1'],[\beta'_2]) + d_{F_1}([\beta'_2], [\ell_2])\\
&\leq & 9 + d_{F_1}(\A',\B).
\end{array}
\]
Consequently,
\[
N \leq d_{F_1}(\A',\B).
\]

Recalling that $G_1 = P\setminus F_1$ is a disc with 2 punctures, by Lemma \ref{Minimal paths}, one of the following occurs:
\begin{itemize}
\item[(a)]  $d_{F_1}(\A',\B) \leq d(\A',\B)$
\item[(b)] Every minimal path from $\A'$ to $\B$ contains $[\ell_1]$.
\end{itemize}

Suppose that (a) holds. Then
\[
N \leq d_{F_1}(\A',\B) \leq d(\A',\B),
\]
as we were hoping.

Now assume that (b) holds. Since every minimal length path from $\A'$ to $\B$ contains $[\ell_1]$, we have:
\[
 d(\A',\B) \geq d(\A', [\ell_1]) + d([\ell_1], \B) \geq d([\ell_1], \B) \geq N
 \]
as desired.
 \qed.

\section{Precedents}

If we work with oriented knots, then the Gordian graph is the 1-skeleton of the ``Gordian complex'' defined by Hirasawa and Uchida \cite{HU}. A number of authors since them have studied the Gordian complex from various points of view. To our knowledge, we are the first to explore the relationship between bridge number, bridge distance and the Gordian graph.

The curve complex was originally defined by Harvey \cite{Harvey} and many of its important properties, including those of subsurface projections, were discovered by Masur and Minsky \cites{MM1, MM2}.

The bridge distance of a knot is a ``knot-version'' of a concept from 3-manifold theory, as we now describe. A (3-dimensional) \defn{handlebody} is the result of attaching solid tubes to a 3-dimensional ball to create an orientable manifold. If $X$ and $Y$ are handlebodies with boundaries having the same genus, we can create a 3-dimensional manifold $M$ by gluing $\boundary X$ to $\boundary Y$ via a homeomorpism $\boundary X \to \boundary Y$. We identify both $X$ and $Y$ with their images in $M$ and call the pair $(X,Y)$ a \defn{Heegaard splitting} of $M$. Every closed, orientable 3-manifold has a Heegaard splitting. If we have a collection of unknotted properly embedded arcs $\kappa \subset X$  and another such collection $\kappa' \subset Y$ and if $\boundary \kappa = \boundary \kappa'$, the surface $\boundary X = \boundary Y$ in $M$ is a \defn{bridge surface} for the knot or link  $\kappa \cup \kappa'$ and $|\kappa|$ is called the \defn{bridge  number} of the surface.

A \defn{disc set} for a Heegaard splitting is the collection of vertices in the curve complex of the Heegaard surface which bound discs all to one side of the Heegaard surface. A Heegaard splitting thus gives rise to two disc sets in the curve complex of the Heegaard surface. Hempel \cite{Hempel} defined the distance of a Heegaard splitting to be the distance in the curve complex between those disc sets. Saito \cite{Saito} adapted that definition to bridge surfaces for knots such that the bridge surface has genus 1 and bridge number 1. Bachman-Schleimer \cite{BS} extended and modified the definition so that it applies to all bridge surfaces. The bridge distance in this paper is a simplification of that of Bachman-Schleimer, and is similar in spirit to Saito's definition.

Lustig and Moriah \cite{LM} prove a result for Heegaard splittings which is similar to our result for knots. They construct high distance Heegaard splittings by performing a single Dehn twist on a gluing map giving a Heegaard splitting of the 3-sphere. Their construction uses Thurston's theory of projective measured foliations to find an appropriate twisting curve.  Finally, our Lemma \ref{diameter} is similar to a result of Tao Li \cite{Li} and Masur-Schleimer \cite{MS} for Heegaard splittings.

\section{Possible Extensions}

A knot has a \defn{$(g,b)$-splitting} if it has a genus $g$ bridge surface of bridge number $b$. This paper studied $(0,b)$ bridge splittings for $b \geq 3$. It should be relatively easy to adapt our results to $(g,b)$ bridge splittings with $g, b \geq 1$. The only real challenge is dealing with the case when $b \leq 2$, as it may be impossible to find a curve in the surface bounding a twice-punctured disc and which is disjoint from a given essential curve. Nevertheless, we make the following conjectures:

\begin{conjecture}
Let $K$ be knot in a 3-manifold $M$ with a $(g,b)$ splitting with $g \geq 1$. Let $(g',b',n)$ be a triple of non-negative integers with $g' \geq g$,  $b' \geq b$. Then there is a knot $K' \subset M$ obtained from $K$ by a single crossing change, such that $K'$ has a $(g',b')$ splitting of distance at least $n$.
\end{conjecture}

The curve complex for a $(0,2)$-bridge sphere of a knot $K \subset S^3$ is not particularly useful as it is disconnected. The Farey graph for the bridge sphere is often a useful substitute. It is the graph whose vertex set consists of essential simple closed curves in the 4-punctured sphere and whose edge set consists of pairs of vertices represented by curves intersecting exactly twice. A result analogous to the one proved in this paper should hold in that setting:

\begin{conjecture}
For each knot $K \subset S^3$ with $\b(K) \leq 2$ and for each $n \in \N$ there exists a 2-bridge knot $K' \subset S^3$ obtained from $K$ by a single crossing change and with $\d(K') \geq n$.
\end{conjecture}

\section{Acknowledgements} The authors are grateful to the American Institute of Mathematics for its support through the SQuaREs program. The third author was supported by NSF Grant DMS-1006369. The fifth author was supported by NSF Grant DMS-1054450. This work was also partially supported by a grant from the Colby College Division of Natural Sciences. The authors also thank the anonymous referees for their careful reading and helpful comments.

\end{document}